\documentclass[11pt,oneside]{amsart}
\usepackage[utf8]{inputenc}

\usepackage{amsmath, comment}
\usepackage{amsthm, amsfonts, amssymb,latexsym, color, tikz}
\usepackage{url}
\usepackage{enumitem}
\title[Geometric Sidon Problems]{Geometric Sidon Problems}
\author{}
\usepackage{amsmath}
\usepackage[a4paper, total={6in, 8in}]{geometry}
\newcommand{\cG}{\mathcal{G}}
\newcommand{\cH}{\mathcal{H}}

\newcommand{\orn}[1]{{{\color{red}#1}}}

\numberwithin{exercise}{subsection}
\newtheorem{lemma}{Lemma}
\newtheorem{obs}{Observation}

\newtheorem{theorem}{Theorem}
\newtheorem*{theorem*}{Theorem}
\newtheorem{cor}{Corollary}

\author{Felix Christian Clemen}
\address{Felix Christian Clemen\\ Department of Mathematics and Statistics, University of Victoria, Victoria, B.C., Canada.}
\email{fclemen@uvic.ca}

\author{Jakob Führer}
\address{Jakob Führer \\ Institute for Algebra, Johannes Kepler University Linz, Linz, Austria.}
\email{jakob.fuehrer@jku.at}

\author{Oliver Roche-Newton}
\address{Oliver Roche-Newton \\ Institute for Algebra, Johannes Kepler University Linz, Linz, Austria.}
\email{o.rochenewton@gmail.com}

\begin{document}
\begin{abstract}
This paper considers geometric problems of the following type: given a point set $P \subset \mathbb R^2$, one seeks a large subset avoiding a prescribed geometric
configuration. Our main result states that, for any $P \subset \mathbb R^2$, there exists a subset $P' \subset P$ with $|P'| \gg |P|^{1/3}$ such that all of the distances determined by $P'$ are distinct. This improves a result of Charalambides. We make heavy use of a result of Li and Postle concerning the independence number of hypergraphs which satisfy some edge distribution conditions, as well as tools from incidence geometry.

\end{abstract}
\maketitle

\section{Introduction}
A recurring theme in discrete geometry is the following subset selection problem:
given a finite point set, how large a subset can one guarantee that avoids a prescribed geometric configuration? Classical examples include the no-three-in-line problem~\cite{Dudeney1917PuzzleWithPawns} and related questions for integer grids, where one seeks large subsets avoiding collinear triples~\cite{Erdos1986}, repeated slopes~\cite{ErdosGrahamRuzsaTaylor1992Dots}, or other small patterns~\cite{JanosikNadorNagySimon2026}. In this paper we study two instances of this theme: selecting subsets with all pairwise distances distinct, and selecting subsets of the integer grid containing no rhombus.

Such problems have a long history in extremal combinatorics and discrete geometry, going back in particular to work of Erdős and Purdy~\cite{ErdosPurdy1971ExtremalGeometry} on forbidden configurations in point sets. They remain an active area of research,
with recent work addressing related subset selection problems for planar grids~\cite{BaloghSubsetSelection,Ghosal2026}. Recently, they have also appeared as benchmark problems in artificial intelligence~\cite{Wagner,georgiev2025mathematical}, due to their simple formulations and computational difficulty.

\subsection*{Notation}

Throughout this paper, the notation $X\gg Y$, $Y \ll X,$ $X=\Omega(Y)$, and $Y=O(X)$ are all equivalent and mean that $X\geq cY$ for some absolute constant $c>0$. The notation $X \gg_k Y$ indicates that $X \geq c_kY$ for some positive constant $c_k$ depending on a parameter $k$. The notation $X \approx Y$ and $X=\Theta (Y)$ denote that both $X \gg Y$ and $X \ll Y$ hold. The notation $\| \cdot \|$ always refers to the Euclidean norm.

\subsection*{Distance Sidon sets in general point sets.}
Given a finite set $P \subset \mathbb R^2$, let $Q(P)$ denote the number of solutions to the equation
\[
\|p-q\|=\|s-t\|
\]
with $p,q,s,t \in P$. We refer to such a solution $(p,q,s,t)$ as a \emph{distance quadruple}. The set $P$ is called a \emph{distance Sidon set} if all solutions are trivial, that is, if $\{p,q\}=\{s,t\}$.

Erd\H{o}s\footnote{The problem was also raised in a much earlier paper of Erd\H{o}s \cite{Er57} which was written in Hungarian. It is indexed as Problem 1208 on Thomas Bloom's Erd\H{o}s Problems Website.} \cite{Er80} raised the question of how large a distance Sidon subset one can guarantee inside an arbitrary finite point set. Using a deletion argument combined with the Guth-Katz \cite{GuKa15} bound on the number of distance quadruples, Charalambides~\cite{charalambides2013note} showed that every $P \subset \mathbb R^2$ contains a subset $P' \subset P$ of size
\[
|P'| \gg \frac{|P|^{1/3}}{(\log |P|)^{1/3}}
\]
that is distance Sidon. We remove the logarithmic factor in this bound and prove the following.

\begin{theorem}\label{thm:main}
For any finite $P \subset \mathbb R^2$, there exists $S \subset P$ such that $|S| \gg |P|^{1/3}$ and $S$ is a distance Sidon set.
\end{theorem}
Several variants of this problem have been studied, in particular the case where $P$ is an integer grid, although it seems to be difficult to make significant use of this extra information. Lefmann and Thiele~\cite{LeTh95} proved that there exists a distance Sidon subset of $[n] \times [n]$ with cardinality $\Omega(n^{2/3})$. Since $|[n] \times [n]|=n^2$, Theorem~\ref{thm:main} gives the same order of magnitude for arbitrary finite planar point sets. A recent blog post of Pohoata and Sheffer~\cite{PohoataSheffer2026} shows that the largest distance Sidon set in $[n] \times [n]$ has cardinality $O \left (n \exp\left( -c \frac{\log n}{\log\log n} \right ) \right )$. See also Conlon, Fox, Gasarch, Harris, Ulrich, and Zbarsky~\cite{conlon2015distinct} for related results concerning distinct volumes in higher dimensions.

\subsection*{Rhombus-free sets in the grid.}

Four points \(a,b,c,d \in \mathbb{R}^2\) form a rhombus if they are the vertices of a parallelogram with all side lengths equal. Equivalently, if in some cyclic ordering
\[
\|a-b\|=\|b-c\|=\|c-d\|=\|d-a\|.
\]
A set $P\subseteq \mathbb{R}^d$ is called \emph{rhombus-free} if it does not contain four points forming a rhombus. J\'anosik, N\'ador, Nagy and Simon~\cite{JanosikNadorNagySimon2026} asked what the maximum size of a rhombus-free subset of the integer grid $[n]^2$ is. They~\cite{JanosikNadorNagySimon2026} proved that there exists a rhombus-free set $S\subseteq [n]^2$ of size
\[
|S| \gg \frac{n^{4/3}}{(\log n)^{1/3}}.
\]
Their argument uses the deletion method together with arithmetic properties of Sidon sets. This bound also follows from a deletion method argument. Indeed, the number of rhombi in $[n]^2$ is $O(n^4\log n)$, since every rhombus contains four isosceles triangles, every isosceles triangle can be completed to at most one rhombus and the number of isosceles triangles in the grid is known to be $O(n^4\log n)$; see for example the proof of Proposition A.1. in \cite{Wagner} for a reference, or Lemma~\ref{lem:rhombicount} for general point sets. A standard random selection argument with probability $p=Cn^{-2/3}\log^{-1/3} n$ then yields a subset of size $\Omega(n^{4/3}(\log n)^{-1/3})$ that is rhombus-free.

Here we improve this bound and prove the following.

\begin{theorem}
\label{thm:rhomingrid}
There exists a rhombus-free set $S\subseteq [n]^2$ of size $|S| \gg n^{4/3}$.
\end{theorem}

Both problems can be phrased as independence number problems in suitable auxiliary hypergraphs. In both proofs, we use tools from combinatorial geometry to establish combinatorial properties of these auxiliary hypergraphs, which in turn yield improved lower bounds on their independence numbers. For the rhombus-free problem, we work with the canonical auxiliary $4$-graph whose edges are rhombi. For the distance Sidon problem, however, we modify the canonical hypergraph to a rank-$4$ hypergraph and bound the relevant combinatorial parameters there.

The paper is organized as follows. In Section~\ref{sec:tools}, we collect the tools used in the proofs, in particular results on independence numbers of hypergraphs and incidence geometry. In Section~\ref{sec:distance}, we prove our main result, Theorem~\ref{thm:main}, on distance Sidon subsets. Finally, in Section~\ref{sec:rhombi}, we prove Theorem~\ref{thm:rhomingrid} on rhombus-free subsets of the grid.

\section{Tools}
\label{sec:tools}

Throughout this paper we will use several results from graph theory and discrete geometry. In this section, we collect all of the external results that will be used.

\subsection{Independence number in hypergraphs}

Let \( \mathcal{H} \) be a rank-\( k \) hypergraph, i.e. a hypergraph where every edge has size at most $k$. For integers \( 2 \leq \ell \leq k \) and a vertex set 
\( S \) with \( 1 \leq |S| < \ell \), we define \( \deg_\ell(S,\mathcal{H}) \) to be the number of edges 
of size \( \ell \) containing \( S \).
The \emph{maximum \( \ell \)-degree} of \( \mathcal{H} \), denoted by \( \Delta_\ell(\mathcal{H}) \), is
\[
\Delta_\ell(\mathcal{H}):= \max_{v \in V(\mathcal{H})} \deg_\ell(\{v\}, \mathcal{H}).
\]
Define
\[
M(\mathcal{H}):
=
\max_{2 \leq \ell \leq k}
\Delta_\ell(\mathcal{H})^{1/(\ell - 1)}.
\]
A standard deletion method argument gives the following lower bound on the independence number. 
\begin{equation*}
\alpha(\mathcal{H})
=
\Omega\!\left(
\frac{|V(\mathcal{H})|}{M(\mathcal{H})}
\right).
\end{equation*}
The \emph{maximum \( (s,\ell) \)-codegree} of \( \mathcal{H} \), denoted by \( \Delta_{s,\ell}(\mathcal{H}) \), is
\[
\Delta_{s,\ell}(\mathcal{H}) := 
\max_{\substack{S \subseteq V(\mathcal{H}) \\ |S| = s}} 
\deg_\ell(S, \mathcal{H}).
\]
Furthermore, let $\Delta_{K_3}(\{v\},\mathcal{H})$ denote the number of copies of
$K_3$ containing $v$ in the graph consisting of the $2$-edges of \(\mathcal H\), and let 
$$\Delta_{K_3}(\mathcal{H}) =\max_{v\in V(\mathcal{H})}\Delta_{K_3}(\{v\},\mathcal{H}).$$

\begin{theorem}[Li, Postle~\cite{LiPostle}]
\label{thm:LiPostle}
Fix \( k \geq 3 \). Let $\mathcal{H}$ be a rank-$k$ hypergraph with maximum $\ell$-degree at most $\Delta_\ell$ for each $2\leq \ell \leq k$. Set $M=\max_{2 \leq \ell \leq k}
\Delta_\ell^{1/(\ell - 1)}$.
Suppose that for all integers \( 2 \leq s < \ell \leq k \), we have $\Delta_{s, \ell}(\mathcal{H}) \leq M^{\ell - s} / f$ and
$\Delta_{K_3}(\mathcal{H}) \leq M^2 / f$. Then,
\[
\chi(\mathcal{H}) \ll 
\max_{2 \leq \ell \leq k} 
\left\{
\left( \frac{\Delta_\ell}{\log f} \right)^{\frac{1}{\ell - 1}}
\right\}.
\]
\end{theorem}
See also \cite{cooper2016coloring} for a version of Theorem~\ref{thm:LiPostle} for uniform hypergraphs, and \cite{clemen2026applications} for an application of this theorem in a geometric setting.
In many applications arising in combinatorial geometry, when the problem is encoded using an auxiliary hypergraph, one often encounters a small number of vertices with unusually large $\ell$-degrees. In such settings, controlling maximum degrees is inconvenient, while average degrees remain well behaved. For \( 2 \leq \ell \leq k \), the \emph{average \( \ell \)-degree} of \( \mathcal{H} \) is defined as
\[
\overline{\Delta}_\ell(\mathcal{H}):
=
\frac{1}{|V(\mathcal{H})|}
\sum_{v \in V(\mathcal{H})} \deg_\ell(\{v\}, \mathcal{H}).
\]
Set
\[
\overline{M}(\mathcal H)
:=
\max_{2 \leq \ell \leq k}
\overline{\Delta}_\ell(\mathcal H)^{1/(\ell - 1)}.
\]
Again, simple probabilistic deletion arguments give us
\begin{equation*}
\alpha(\mathcal{H})
=
\Omega\!\left(
\frac{|V(\mathcal{H})|}{\overline M(\mathcal{H})}
\right).
\end{equation*}
It is useful to formulate a variant of Theorem~\ref{thm:LiPostle} in terms of average degrees, and we state such an average-degree version below. This comes at the cost that the conclusion only gives an independence-number
bound, rather than a chromatic-number bound; this is sufficient for our
applications.

\begin{cor}
\label{corl:ind}
Fix \( k \geq 3 \). Let $\mathcal{H}$ be a rank-$k$ hypergraph with average $\ell$-degree at most $\overline{\Delta}_\ell$ for each $2\leq \ell \leq k$. Set $\overline{M}:=\max_{2 \leq \ell \leq k}
\overline{\Delta}_\ell^{1/(\ell - 1)}$.
Suppose that for all integers \( 2 \leq s < \ell \leq k \), $\Delta_{s, \ell}(\mathcal{H}) \leq \overline{M}^{\ell - s} / f$ and $\Delta_{K_3}(\mathcal{H}) \leq \overline{M}^2 / f$.
Then,
\[
\alpha(\mathcal{H})
\gg_k
|V(\mathcal{H})| \cdot
\min_{2 \leq \ell \leq k}
\left\{
\left(
\frac{\log f}{\overline{\Delta}_\ell}
\right)^{1/(\ell - 1)}
\right\}.
\]
\end{cor}

\begin{proof}[Proof of Corollary~\ref{corl:ind}]
 Let $\mathcal{H}$ be a rank-$k$ hypergraph with average $\ell$-degree at most $\overline{\Delta}_\ell$ for each $2\leq \ell \leq k$, i.e. $\overline\Delta_\ell(\mathcal{H})\leq \overline{\Delta}_\ell$.  Set $\overline{M}:=\max_{2 \leq \ell \leq k}
\overline{\Delta}_\ell^{1/(\ell - 1)}$ and $V = V(\mathcal{H})$. 
For each $2 \le \ell \le k$, define
\[
V_\ell:=\left\{v \in V :\deg_\ell(\{v\}, \mathcal{H})>2(k-1)\, \overline{\Delta}_\ell(\mathcal{H})
\right\}.\]
By definition of average $\ell$-degree,
\[
|V|\, \overline{\Delta}_\ell(\mathcal{H})
>|V_\ell| \cdot 2(k-1)\,\overline{\Delta}_\ell(\mathcal{H}),\]
hence $|V_\ell| < \frac{|V|}{2(k-1)}$.
Now define
\[
U:
=
V
\setminus
\bigcup_{2 \le \ell \le k} V_\ell.
\]
Since there are exactly $k-1$ possible values of $\ell$,
\[
\left|
\bigcup_{2 \le \ell \le k} V_\ell
\right|\leq \sum_{\ell=2}^k|V_\ell|
<
(k-1)\frac{|V|}{2(k-1)}
=
\frac{|V|}{2}.
\]
Therefore $|U| \ge |V|/2$. Let $\mathcal{H}'=\mathcal{H}[U]$ be the induced subhypergraph of $\mathcal{H}$ on $U$. By construction $$\Delta_\ell(\mathcal{H}')
\le
2(k-1)\,\overline{\Delta}_\ell(\mathcal{H})\leq 2(k-1)\overline{\Delta}_\ell
\quad
\text{for all } 2 \le \ell \le k.$$
Deleting vertices cannot increase codegrees, so for all 
$2 \le s < \ell \le k$,
\begin{equation}
\label{eq:codegreeH'H}
\Delta_{s,\ell}(\mathcal{H}')
\le
\Delta_{s,\ell}(\mathcal{H})
\qquad \text{and} \qquad 
\Delta_{K_3}(\mathcal{H}')
\le
\Delta_{K_3}(\mathcal{H}).
\end{equation}
Define $\Delta_\ell:=2(k-1)\overline{\Delta}_\ell$ for $2\leq \ell \leq k$ and set $M=\max_{2 \leq \ell \leq k}\Delta_\ell^{1/(\ell-1)}$. Note that $\overline{M}\leq M$.
Therefore, using \eqref{eq:codegreeH'H} and the hypothesis of this corollary,
\begin{equation*}
\Delta_{s,\ell}(\mathcal{H}')\leq 
\Delta_{s,\ell}(\mathcal{H})
\le
\frac{\overline{M}^{\ell-s}}{f}\leq \frac{M^{\ell-s}}{f} 
\qquad \text{and} \qquad \Delta_{K_3}(\mathcal{H}')
\leq 
\Delta_{K_3}(\mathcal{H})
\le
\frac{\overline{M}^2}{f}\leq \frac{M^2}{f}.
\end{equation*}
Therefore, the hypotheses of Theorem~\ref{thm:LiPostle} hold for $\mathcal{H}'$. By applying Theorem~\ref{thm:LiPostle}, we obtain
\[
\chi(\mathcal{H}')
\ll_k
\max_{2 \le \ell \le k}
\left\{
\left(
\frac{\Delta_\ell}{\log f}
\right)^{1/(\ell - 1)}
\right\}.
\]
Since every proper colouring of \(\mathcal H'\) has a colour class of size at least
\(|U|/\chi(\mathcal H')\), this implies
\[
\alpha(\mathcal{H}')
\gg_k
|U|
\min_{2 \le \ell \le k}
\left\{
\left(
\frac{\log f}{\Delta_\ell}
\right)^{1/(\ell - 1)}
\right\}
.
\]
Since $|U| \ge |V|/2$, $\Delta_\ell=\Theta(\overline{\Delta}_\ell)$ and $\alpha(\mathcal{H}) \ge \alpha(\mathcal{H}')$,
the desired lower bound on $\alpha(\mathcal{H})$ follows.
\end{proof}

\subsection{Tools from discrete geometry}

A key tool for the proof of Theorem \ref{thm:main} is a celebrated result of Guth and Katz~\cite{GuKa15}, which solved the Erd\H{o}s distinct distances problem up to logarithmic factors. In particular, we need the following result from \cite{GuKa15} which bounds the distance energy. 

\begin{theorem}[Guth and Katz~\cite{GuKa15}] \label{thm:GK}
 For any $P \subset \mathbb R^2$,
 \[
| \{ (p,q,p',q') \in P^4 : \| p-q\|=\|p'-q' \| \}|\ll |P|^3 \log|P|.
 \]
\end{theorem}

We will need some results from incidence geometry, including the classical Szmer\'{e}di-Trotter Theorem.

\begin{theorem}[Szemer\'{e}di-Trotter Theorem]
\label{ST}
For any finite $P \subset \mathbb R^2$ and any finite set $L$ of lines in $\mathbb R^2$,
\[
I(P,L) := | \{ (p,\ell) \in P \times L : p \in \ell \}| \ll |P|^{2/3}| L|^{2/3}   + |P| + |L|.
\]
\end{theorem}

A well-known consequence of (a variant of) the Szemer\'{e}di-Trotter Theorem is the following bound \cite{SST83} for the maximum number of occurrences of a fixed distance, an infamous open problem in discrete geometry. See also \cite{OpenAI2026PlanarUnitDistanceBlog,Sawin2026UnitDistance} for the very recent exceptional breakthrough by OpenAI.

\begin{theorem} [Spencer, Szemer\'edi, Trotter \cite{SST83}]\label{thm:unit}
For any finite $P \subset \mathbb R^2$ and any $d \geq 0$,
\[
|\{(p,q) \in P \times P :  \|p-q\|=d \}| \ll |P|^{4/3}.
\]
\end{theorem}

We will use a variant of the Szemer\'{e}di-Trotter Theorem for point-circle incidences, due to Marcus and Tardos \cite{MaTa06}.

\begin{theorem}[Marcus and Tardos \cite{MaTa06}]\label{thm:circleincidences}
Let $P \subset \mathbb R^2$ be finite and let $ \mathcal C$ be a finite family of circles in $\mathbb R^2$. Then
\[
 I(P, \mathcal C) \ll |P|^{2/3}| \mathcal C|^{2/3} + |P|^{6/11}|\mathcal C|^{9/11} \log^{2/11} |P| + |P| + | \mathcal C|.
\]

\end{theorem}

As is standard in incidence theory, the incidence theorem above can be used to bound the number of $k$-rich circles determined by a point set, as follows.

\begin{cor} \label{cor:richcircles}
Let $P$ be a point set and $k \geq 3$. Let $\mathcal C_k$ denote the set of all circles which contain at least $k$ points from $P$. Then
\[
|\mathcal C_k| \ll \frac{|P|^3}{k^{11/2}}\log|P| + \frac{|P|^2}{k^{3}}+ \frac{|P|}{k}.  
\]
\end{cor}

The following simple corollary will be used twice in this paper and so we give a quick proof here.

\begin{cor} \label{cor:norichcircles}
    Let $P \subset \mathbb R^2$ be finite and sufficiently large. There exists a subset $P' \subset P$ such that $|P'| \geq |P|/2$ and any circle $C$ centred at a point $p \in P'$ satisfies $|C \cap P'| \leq |P|^{1/2}$.
\end{cor}

\begin{proof}
    Corollary \ref{cor:richcircles} tells us that there are at most $O(|P|^{1/2})$ circles $C$ such that $|C \cap P| \geq |P|^{1/2}$. If such a rich circle has a centre in $P$, we remove this centre from the point set. We remove at most $c|P|^{1/2} \leq |P|/2$ points to arrive at our final set $P' \subset P$. In the inequality above we have used the assumption that $P$ is sufficiently large.
\end{proof}

For a point set $P \subset \mathbb R^2$ and $d>0$, let $r_P(d)$ denote the number of representations of $d$ as a distance in $P$, that is,
\[
r_P(d):=| \{(p,q) \in P \times P : \|p-q\|=d \}|.
\]
We use Theorem \ref{thm:circleincidences} to bound the number of distances with many representations.

\begin{lemma} \label{lem:richd}
Let $P \subset \mathbb R^2$ and let $t \geq 1$. Then
\[
|\{ d : r_P(d) \geq t \}| \ll \frac{|P|^{\frac{15}{2}}\log|P|}{t^{\frac{11}{2}}}.
\]
\end{lemma}
\begin{proof}
Let $X:= \{ d : r_P(d) \geq t \}$. We restrict our attention to the case when $t\geq C|P|$ for a sufficiently large absolute constant $C$. In case $t < C|P|$, the trivial bound $|X| \leq \frac{|P|^2}{t}$ is sufficiently strong to imply the lemma. We can also restrict to the case $t \leq c|P|^{4/3}$ for a suitable constant $c$, since otherwise Theorem \ref{thm:unit} implies that $X$ is empty. Let $C_p(d)$ denote the circle with radius $d$ centred at $p \in \mathbb R^2$. Define the set of circles
\[
\mathcal C:= \{ C_p(d) : p \in P,\ d \in X \}
\]
and note that $|\mathcal C|=|P||X|$. Since every pair $(p,q)$ such that $\|p-q\| \in X$ contributes an incidence to the count $I(P, \mathcal C)$, it follows from Theorem \ref{thm:circleincidences} that
\[
     |X|t \leq I(P, \mathcal C) \ll |P|^{4/3}|X|^{2/3} + |P|^{15/11}|X|^{9/11} \log^{2/11} |P| + |P||X|.
\]
By taking $C$ sufficiently large, we can ensure that the third term is not dominant and thus simplify this to
\[
    |X|t \ll |P|^{4/3}|X|^{2/3} + |P|^{15/11}|X|^{9/11} \log^{2/11}|P| .
\]
A rearrangement of this latter inequality gives 
\[
|X| \ll \frac{|P|^4}{t^3} + \frac{|P|^{\frac{15}{2}}\log|P|}{t^{\frac{11}{2}}} \ll  \frac{|P|^{\frac{15}{2}}\log|P|}{t^{\frac{11}{2}}},
\]
where the last inequality uses the assumption that $t \leq c|P|^{4/3}$.
\end{proof}
We will also need a bound for the number of isosceles triangles determined by a point set. The following result can be derived from Theorem \ref{thm:circleincidences}, and is strong enough for our purposes.
\begin{lemma} \label{lem:isosceles}
 Let $P \subset \mathbb R^2$ be a finite point set. Then
 \[
 T(P):= |\{(p,q_1,q_2) \in P^3 : \|p-q_1\|=\|p-q_2\| \}| \ll |P|^{2+\frac{2}{9}}\log^{\frac{2}{9}}|P|.
 \]
\end{lemma}
\begin{proof}
 Let $\Delta \geq 3$ be a parameter to be specified later. For each $p \in P$, let $\mathcal C(p)$ denote the minimal set of circles centred at $p$ and covering $P$, and let $\mathcal C= \cup_{p \in P} \ \mathcal{C}(p)$. Note that
\begin{align*}
T(P)&= \sum_{p \in P} \sum_{C \in \mathcal C(p)} |C \cap P|^2
\\&= \sum_{p \in P} \sum_{C \in \mathcal C(p): |C \cap P| \leq \Delta} |C \cap P|^2 +  \sum_{p \in P} \sum_{C \in \mathcal C(p): |C \cap P|>\Delta} |C \cap P|^2
\\&\leq \Delta \sum_{p \in P} \sum_{C \in \mathcal C(p)} |C \cap P| +   \sum_{C \in \mathcal C: |C \cap P|>\Delta} |C \cap P|^2
\\& \leq \Delta |P|^2 + \sum_{j \in \mathbb N : 2^{j-1}\Delta \leq |P|} \sum_{ \substack{C \in \mathcal C: \\ 2^{j-1}\Delta \leq |C \cap P|<2^j\Delta}} |C \cap P|^2
\\& \ll \Delta |P|^2 + \sum_{j \in \mathbb N : 2^{j-1}\Delta \leq |P|}  \left ( \frac{|P|^3}{(2^j\Delta)^{7/2}}\log|P| + \frac{|P|^2}{2^j\Delta}+ 2^j\Delta|P| \right )
\\ & \ll \Delta |P|^2 + \frac{|P|^3}{\Delta^{7/2}}\log|P|,
\end{align*}
where the third inequality above follows from an application of Corollary \ref{cor:richcircles}. Setting $\Delta= |P|^{2/9} \log^{2/9}|P|$ completes the proof.
\end{proof}

Note that a better estimate for the number of isosceles triangles determined by $P$ can be found in work of Pach and Tardos~\cite{PaTa02}, where a bound of $O(|P|^{2.137})$ was proven. We give a weaker bound above in an effort to reduce the number of external results we need to call upon in this paper. The important thing for the forthcoming application of Lemma \ref{lem:isosceles} is that the exponent is strictly less than $7/3$.

\section{Distance Sidon sets}
\label{sec:distance}

The goal of this section is to prove Theorem \ref{thm:main}. The idea for this proof is to construct a hypergraph encoding distance quadruples of $P$ and then apply Corollary \ref{corl:ind} to lower bound the independence number of the hypergraph. Before we are ready to do this, we need to make some preparations by passing to a large subset of $P$ such that the graph defined using this subset satisfies the conditions of Corollary \ref{corl:ind}, and this is the purpose of the next lemma. In the following statement (and throughout the rest of this paper), the notation $b(p,q)$ refers to the perpendicular bisector of two distinct points $p,q \in \mathbb R^2$.

\begin{lemma} \label{lem:subset}
Let $P \subset \mathbb R^2$ with $|P|=n$. There exists $P' \subset P$ such that $|P'| \gg n$ and $P'$ satisfies all of the following conditions:
\begin{enumerate}[label=(\alph*)]
    \item Let $T(p)$ denote the number of isosceles triangles in $P'$ which contain $p$. Then, for all $p \in P'$
    \[
T(p) \ll n^{11/9} \log^{2/9}n.
    \]
    \item For all $p \in P'$,
    \[
    | \{ q \in P' : |b(p,q) \cap P'| \geq n^{\frac{2}{3}- \frac{1}{90}} \}| \ll n^{\frac{17}{30}} \log^{\frac{2}{9}}n.
    \]
    \item For all $p \in P'$
    \[
    |\{ q \in P' : r_{P'}(\|p-q\|) > n^{\frac{4}{3}-\frac{1}{99}} \}| \ll n^{\frac{5}{9}} \log n.
    \]
    \item For all $p \in P'$ and any circle $C$ centred at $p$, $|C \cap P'| \leq n^{1/2}$.
\end{enumerate}
\end{lemma}

\begin{proof}

Let us assume for simplicity that $n$ is a multiple of $8$. If this is not the case then we can achieve it by deleting at most $7$ elements. We begin with (a).  By Lemma \ref{lem:isosceles},
\begin{equation} \label{sum41}
\sum_{p \in P} T(p) \ll |P|^{2+\frac{2}{9}} \log^{\frac{2}{9}} |P|.
\end{equation}
Label the elements of $P=\{p_1,\dots , p_n\}$ so that $T(p_1) \leq T(p_2) \leq \dots \leq T(p_n)$. Note that $T(p_{n/2}) \leq C|P|^{11/9} \log^{2/9} |P| $ for some absolute constant $C$, since otherwise
\[
\sum_{i= \frac{n}{2}+1}^n T(p_i) \geq \frac{C}{2}|P|^{2+\frac{2}{9}} \log^{\frac{2}{9}}|P|
\]
and this contradicts \eqref{sum41} provided that $C$ is sufficiently large. Let $P_1 = \{p_1, \dots, p_{n/2} \}$, and so each element of $P_1$ is involved in at most $O(n^{11/9} \log^{2/9} n)$ isosceles triangles determined by $P_1$. For the rest of this proof, we will pass to further subsets of $P_1$, with the final subset $P'$ having size $\Omega(n)$. This will complete the proof of (a).


We say that the pair $(p,q) \in P \times P$ is \textit{bad} if $|b(p,q) \cap P_1| \geq n^{\frac{2}{3}-\frac{1}{90}}$. It follows from Lemma \ref{lem:isosceles} that 
\[
\{ \text{bad pairs} \}\cdot n^{\frac{2}{3}-\frac{1}{90}} \leq T(P_1) \ll n^{2+\frac{2}{9}}\log ^{\frac29}n
\]
and thus there are $O(n^\frac{47}{30}\log ^{\frac29}n)$ bad pairs. Now let $B(p)$ denote the number of bad pairs involving $p$. Again, we can pass to a subset $P_2 \subset P_1$ such that $|P_2| \geq |P_1|/2=n/4$ and with each element $p \in P_2$ being involved in $O(n^{\frac{17}{30}}\log ^{\frac29}n)$ bad pairs. Any sufficiently large subset of $P_2$ will therefore satisfy condition (b) of the lemma.

Now we turn to (c). Apply Lemma \ref{lem:richd} with $t=n^{\frac{4}{3}-\frac{1}{99}}$. It follows that
\[
|\{d : r_{P_2}(d) \geq n^{\frac{4}{3}-\frac{1}{99}}\}| \ll n^{2/9}\log n.
\]
It then follows from the unit distance bound Theorem \ref{thm:unit} that
\begin{equation} \label{pairs}
|\{(p,q) \in P_2 \times P_2 : r_{P_2}(\|p-q\|) \geq n^{\frac{4}{3}-\frac{1}{99}}\}| \ll n^{14/9}\log n.
\end{equation}
Now, fix $p \in P_2$ and define
\[
m(p):= | \{ q \in P_2 : r_{P_2}(\|p-q\|) \geq n^{\frac{4}{3}-\frac{1}{99}}\}|.
\]
It follows from \eqref{pairs} that 
\[
\sum_{p \in P_2} m(p) \ll n^{\frac{14}{9}}\log n.
\]
We can then repeat the previous argument to choose a subset $P_3 \subset P_2$ consisting of the first half of the elements of $P_2$ when we order in increasing value of $m(p)$. Thus $|P_3| \geq |P_2|/2 \geq n/8$ and all $ p \in P_3$ satisfy $m(p) \ll n^{\frac{5}{9}}\log|P|$. In particular, condition (c) is satisfied for any subset $P' \subset P_3$ with $|P'| \gg |P_3|$.

We make one final refinement of the point set to ensure that condition (d) is satisfied. An application of Corollary \ref{cor:norichcircles} gives us our final set $P' \subset P_3$ with $|P'| \gg |P_3| \gg |P|$ and such that no circle centred at an element of $P'$ contains more than $n^{1/2}$ elements of $P'$.\end{proof}
We are now ready to prove Theorem \ref{thm:main}.

\begin{proof}[Proof of Theorem \ref{thm:main}]
Let $P' \subset P$ be the set given by Lemma \ref{lem:subset}. 
We construct a rank-$4$ hypergraph $\cH$ to encode distance quadruples, although with some care taken to treat certain problematic quadruples separately. The edges of $\cH$ are defined as follows.

\begin{enumerate}
    \item The $4$-edges of $\cH$ are $4$-sets $E\subset P'$ such that there is a partition \(E=\{p,q\}\cup\{s,t\}\) with $\|p-q\|=\|s-t\|=d$ and $r_{P'}(d) \leq n^{\frac{4}{3} - \frac{1}{99}}$.
    \item The $3$-edges of $\cH$ are the $3$-sets \(E\subset P'\) such that
there exist distinct \(p,q,s\in E\) with $\|p-q\|=\|p-s\|$ and $|b(q,s) \cap P'| \leq n^{\frac{2}{3} - \frac{1}{90}}$.
    \item If $\|p-q\|=d$ and $r_{P'}(d) > n^{\frac{4}{3} - \frac{1}{99}}$ then $\{p,q\}$ forms a $2$-edge.
    \item  If $|b(p,q) \cap P'| > n^{\frac{2}{3} - \frac{1}{90}}$ then $\{p,q\}$ forms a $2$-edge.
\end{enumerate}

The important idea here is that the hypergraph is defined in such a way that an independent set in the hypergraph is a distance Sidon set. 
\begin{obs}
\label{indepinH}
Every independent set in \(\mathcal H\) is a distance Sidon set.
\end{obs}

\begin{proof}
Let $I\subseteq P'$ be independent. Suppose that two distinct pairs in $I$
determine the same distance. If the two pairs are disjoint, then either the
corresponding distance is represented at most $n^{\frac43-\frac1{99}}$ times
in $P'$, in which case these four points span a $4$-edge, or it is represented
more than $n^{\frac43-\frac1{99}}$ times, in which case one of the pairs is a
$2$-edge. Both contradict independence.

If the two pairs share a point, then for some distinct $p,q,s\in I$ we have
\[
\|p-q\|=\|p-s\|.
\]
If $|b(q,s)\cap P'|\leq n^{\frac23-\frac1{90}}$, then $\{p,q,s\}$ is a
$3$-edge; otherwise $\{q,s\}$ is a $2$-edge. Again this contradicts
independence. Hence all distances determined by $I$ are distinct.
\end{proof}

It remains to show that $\mathcal{H}$ contains an independent set of size $\Omega(n^{1/3})$. Theorem \ref{thm:GK} ensures that $\overline{\Delta}_4(\mathcal{H}) \ll n^2 \log n$. Lemma~\ref{lem:isosceles} ensures that $\overline{\Delta}_3(\mathcal{H})\ll |P|^{1+\frac{2}{9}}\log^{\frac{2}{9}}|P|$. A bound for $\overline{\Delta}_2(\mathcal{H})$ is given by conditions (b) and (c) of Lemma \ref{lem:subset}. Each vertex $p \in V(H)$ is involved in at most  $O(n^{\frac{17}{30}} \log^{\frac29}n + n^{\frac59} \log n)=O(n^{\frac{17}{30}} \log^{\frac29}n)$ $2$-edges.  It follows that
\[
\overline{\Delta}_2(\mathcal{H}) \ll n^\frac{17}{30}\log ^{\frac29}n
\]
Therefore we can choose 
\begin{equation*}
\overline{\Delta}_4=\Theta(n^2 \log n), \quad \overline{\Delta}_3=\Theta(n^{\frac{11}{9}} \log^{\frac{2}{9}} n), \quad \text{and} \quad \overline{\Delta}_2=\Theta(n^\frac{17}{30}\log ^{\frac29}n)
\end{equation*}
such that $\overline{\Delta}_\ell(\mathcal{H})\leq \overline{\Delta}_\ell$ for $\ell\in \{2,3,4\}$. As in the statement of Corollary \ref{corl:ind}, let
\[
\overline{M}=\max \{
\overline{\Delta}_2, \overline{\Delta}_3^{1/2}, \overline{\Delta}_4^{1/3} \}.
\]
  We conclude that $\overline M=\Theta( n^{\frac{2}{3}} \log^{\frac{1}{3}} n)$. We will now verify that the following four inequalities hold with $f= n^{\frac{1}{100}}$:
\begin{align}
 \Delta_{3,4}(\cH) &\leq \frac{\overline M}{f} \label{term1},
 \\ \Delta_{2,4}(\cH) &\leq \frac{\overline M^2}{f} \label{term2},
  \\ \Delta_{2,3}(\cH) &\leq \frac{\overline M}{f} \label{term3},
   \\ \Delta_{K_3}(\cH) &\leq \frac{\overline M^2}{f} \label{term4}.
\end{align}
We begin with \eqref{term1}. For fixed $p,q,s \in P'$, we need to bound the number of $t \in P'$ such that $\{p,q,s,t \}$ forms a distance quadruple. It follows from condition (d) of Lemma \ref{lem:subset} that there are $O(n^{1/2})$ possible choices for $t \in P'$, and thus $\Delta_{3,4}(\cH) \ll n^{1/2}$. This verifies \eqref{term1}.

Next we consider $\Delta_{2,4}(\cH)$. Fix $p,q \in P'$. We need to bound the number of sets $\{s,t \} \subset P'$ such that $\{p,q,s,t\}$ forms a distance quadruple satisfying the additional restriction described in the definition of the $4$-edges of $\cH$. There are two kinds of solutions that may occur, namely 
\[
d=\|p-q\|=\|s-t\|
\]
and
\begin{equation} \label{type2}
\|p-s\|=\|q-t\|.
\end{equation}
For the first of these, we may assume that $r_{P'}(d) \leq n^{\frac{4}{3}- \frac{1}{99}}$, since otherwise there are no $4$-edges corresponding to such solutions. It therefore follows that there are at most $n^{\frac{4}{3}- \frac{1}{99}}$ choices for $s$ and $t$. 

Now consider how many choices of $s,t$ exist such that \eqref{type2} holds. This quantity is equal to
\begin{align*}
\sum_d |C_p(d) \cap P'| |C_q(d) \cap P'| &\leq \left ( \sum_d |C_p(d) \cap P'|^2 \right)^{1/2} \left ( \sum_d |C_q(d) \cap P'|^2 \right)^{1/2}
\\& = | \{ (s,t) \in P' \times P' : \|p-s\|=\|p-t\|\}|^{1/2}
\\& \cdot|\{ (s,t) \in P' \times P' : \|q-s\|=\|q-t\|\}|^{1/2}
\\& \ll n^{11/9} \log^{2/9}n.
\end{align*}
The first inequality above is an application of Cauchy-Schwarz, while the last uses condition (a) of Lemma \ref{lem:subset}. We have therefore shown that $\Delta_{2,4}(\mathcal{H}) \ll n^{\frac{4}{3}- \frac{1}{99}}$, which verifies \eqref{term2}.

Next, we consider $\Delta_{2,3}(\cH)$. Fix distinct $p,q \in P'$. We need to bound the number of choices for $s$ such that the set $ \{p,q,s\}$ forms an isosceles triangle satisfying the additional constraint in the definition of the $3$-edges of $\cH$. Again, there are two equations to consider, namely
\[
\|p-q\|=\|p-s\|
\]
and
\begin{equation} \label{isoscelescase}
\|p-s\|=\|q-s\|.
\end{equation}
For the first of these equations, it follows from condition (d) of Lemma \ref{lem:subset} that there are $O(n^{1/2})$ choices for $s \in P'$. For the second of these equations, we are asking for the number of solutions to $s \in b(p,q)$. We may assume that $|b(p,q) \cap P'| \leq n^{\frac{2}{3}-\frac{1}{90}}$, since otherwise there are no $3$-edges in $\cH$ arising from such solutions. It therefore follows that there are at most $n^{\frac{2}{3}-\frac{1}{90}}$ choices for $s$ which satisfy \eqref{isoscelescase}, and thus $\Delta_{2,3}(\cH) \ll n^{\frac{2}{3}-\frac{1}{90}}$, which verifies \eqref{term3}.

Finally, we need to verify \eqref{term4}. For fixed $p \in P'$, we need to bound the number of $q,s \in P'$ such that each of $\{p,q\}, \{p,s\}$ and $\{q,s \}$ are $2$-edges in $\cH$. It follows from conditions (b) and (c) of Lemma \ref{lem:subset} that $p$ is contained in $O(n^{\frac{17}{30}} \log^\frac{2}{9}n)$ $2$-edges. In particular,
\[
\Delta_{K_3}(\cH) \leq \max_{p \in P'} |\{ q \in P' : \{p,q \} \text{ is a $2$-edge}\}|^2  \ll n^{\frac{17}{15}} \log^{\frac{4}{9}} n.
\]
We have now verified \eqref{term4}, and we have therefore verified that all of the conditions of Corollary \ref{corl:ind} hold. It follows that
\[
\alpha(\mathcal{H})
\gg
|V(\mathcal{H})|
\min_{2 \leq \ell \leq k}
\left\{
\left(
\frac{\log f}{\overline{\Delta}_\ell}
\right)^{1/(\ell - 1)}
\right\}
\gg
n^{1/3}.
\]
Now let $I$ be a subset of $P$ which forms an independent set in $\mathcal{H}$ of maximal size. By Observation~\ref{indepinH}, it follows that $I$ is a distance Sidon set, and the proof is complete.
\end{proof}

\section{Rhombus-free sets}
\label{sec:rhombi}
In this section we prove Theorem~\ref{thm:rhomingrid}. Recall that a rhombus consists of four distinct points $a,b,c,d \in \mathbb R^2$ such that for some relabelling
\[
\|a-b\|=\|b-c\|=\|c-d\|=\|d-a\|.
\]
We refer to the pairs $\{a,c\}$ and $\{b,d\}$ as the diagonals of the rhombus. The other four pairs formed from the set $\{a,b,c,d\}$ are said to be adjacent.

Recall that the number of rhombi in the grid $[n]^2$ is $O(n^4 \log n)$. More generally, the same asymptotic upper bound holds for arbitrary finite planar point sets. We include the proof of this more general statement here, since it may be useful for extending Theorem~\ref{thm:rhomingrid} to general planar point sets.

\begin{lemma} \label{lem:rhombicount}
Let $P\subseteq \mathbb{R}^2$. Then $P$ spans at most $O(|P|^2 \log |P|)$ many rhombi. 
\end{lemma}
\begin{proof}

Every rhombus defines two perpendicular lines $\ell_1$ and $\ell_2$, each passing through two opposing vertices. Let $m$ be the intersection point of $\ell_1$ and $\ell_2$ and let $\ell_i^+$ be one of the connected components of $\ell_i\setminus\{m\}$, for $i\in\{1,2\}$. Each pair of points $(a,b)\in\ell_1^+\times \ell_2^+$ uniquely defines a rhombus $\{a,b,2m-a,2m-b\}$, see Figure~\ref{rhombusfigure}.
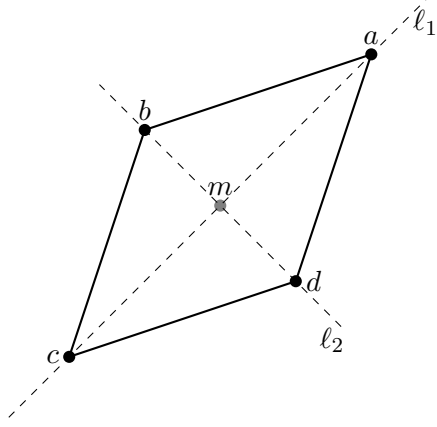
\begin{figure}[h]
\begin{tikzpicture}[scale=1]

  \coordinate (A) at (0,0);
  \coordinate (B) at (3,1);
  \coordinate (C) at (4,4);
  \coordinate (D) at (1,3);
  \coordinate (E) at (2,2);

  \draw[thick] (A) -- (B) -- (C) -- (D) -- cycle;

  \node[left] at (A) {$c$};
  \node[right] at (B) {$d$};
  \node[above] at (C) {$a$};
  \node[above]  at (D) {$b$};
  \node[above]  at (E) {$m$};

  \fill (A) circle (0.08);
  \fill (B) circle (0.08);
  \fill (C) circle (0.08);
  \fill (D) circle (0.08);
  \fill[gray] (E) circle (0.08);

\path (A) -- (C) coordinate[pos=-0.2](A') coordinate[pos=1.2](C');
\draw[dashed] (A') -- (A) -- (C) -- (C')node[pos=0.9, below] {$\ell_1$};

\path (B) -- (D) coordinate[pos=-0.3](B') coordinate[pos=1.3](D');
\draw[dashed] (B') -- (B)node[pos=0.2, below] {$\ell_2$} -- (D) -- (D');

\end{tikzpicture}
\caption{A rhombus with center $m$ and diagonals $\ell_1,\ell_2$.}
\label{rhombusfigure}
\end{figure}
Let $L$ be the set of lines passing through at least two points of $P$ and let $$D:=\{(\ell_1,\ell_2)\in L^2\;|\;  \ell_1\perp \ell_2 \land |\ell_1\cap P|\geq |\ell_2\cap P|\}.$$
Further, let $R(P)$ denote the number of rhombi defined by $P$. For a line $\ell$ and point $x$, write $r_{\ell}(x):=|\{\{u,v\}\in \binom{\ell\cap P}{2}\;|\;u+v=2x\}|$. For two distinct non parallel lines $\ell_1$ and $\ell_2$, let $m_{\ell_1,\ell_2}$ denote the point of intersection of the two lines. Then
    $$R(P)\leq \sum_{\substack{(\ell_1,\ell_2)\in D }}r_{\ell_1}(m_{\ell_1,\ell_2})r_{\ell_2}(m_{\ell_1,\ell_2}).$$
Given $k\geq 1$, let $$R_k:=\sum_{\substack{(\ell_1,\ell_2)\in D\\  \\2^k\leq |\ell_1 \cap P|<2^{k+1}}}r_{\ell_1}(m_{\ell_1,\ell_2})r_{\ell_2}(m_{\ell_1,\ell_2}).$$
Put \(t=2^k\). First suppose that \(t\leq |P|^{1/2}\). By the Szemer\'{e}di-Trotter Theorem (Theorem~\ref{ST}), the number of lines \(\ell_1\in L\) satisfying
\(|\ell_1\cap P|\geq t\) is \(O(|P|^2/t^3)\). For a fixed such line
\(\ell_1\), every pair \(\{a,c\}\in \binom{\ell_1\cap P}{2}\)
determines the midpoint $m=(a+c)/2$, and hence determines the unique line \(\ell_2\) through \(m\) perpendicular
to \(\ell_1\). Since \((\ell_1,\ell_2)\in D\), we have $|\ell_2\cap P|\leq |\ell_1\cap P|<2t$, and therefore \(r_{\ell_2}(m)\leq 2t\). Thus

$$R_k\leq\sum_{\substack{\ell_1\in L\\ t\leq |\ell_1 \cap P|<2t}}\sum_{\{a,c\}\in{\binom{\ell_1\cap P}{2}}}r_{\ell_2}(m_{\ell_1,\ell_2})\ll \frac{|P|^2}{t^3} \binom{2t}{2}2t\ll |P|^2.$$
If $t> |P|^{1/2}$, then the number of lines $\ell_1\in L$ such that $|\ell_1\cap P|\geq t$ is at most $O(|P|/t)$. For fixed \(\ell_1\), every rhombus counted by \(R_k\) is determined by choosing one vertex
\(a\in \ell_1\cap P\) and one adjacent vertex
\(b\in P\setminus \ell_1\). Therefore,
$$R_k \ll\sum_{\substack{\ell_1\in L\\ t\leq |\ell_1 \cap P|<2t}}\sum_{a\in\ell_1\cap P}|P\setminus \ell_1|\leq \sum_{\substack{\ell_1\in L\\ t\leq |\ell_1 \cap P|<2t}} 2t |P|\ll |P|^2.$$
In total, the number of rhombi is 
\begin{align*}   
R(P)&\leq \sum_{\substack{(\ell_1,\ell_2)\in D }}r_{\ell_1}(m_{\ell_1,\ell_2})r_{\ell_2}(m_{\ell_1,\ell_2})
    =\sum_{k=1}^{\lceil \log_2|P| \rceil}\sum_{\substack{(\ell_1,\ell_2)\in D\\   \\2^k\leq |\ell_1 \cap P|<2^{k+1}}}r_{\ell_1}(m_{\ell_1,\ell_2})r_{\ell_2}(m_{\ell_1,\ell_2})  \\
    &=\sum_{k=1}^{\lceil \log_2|P| \rceil}R_k   
    \ll  |P|^2\log|P|,
    \end{align*}
completing the proof.
\end{proof}
We remark that Lemma~\ref{lem:rhombicount} may be of independent interest. 
In contrast, the analogous bound for isosceles triangles is open: 
The best known upper bound is \(O(n^{2.137})\), due to Pach and Tardos~\cite{PaTa02}. In light of the remarkable recent breakthrough constructions giving many occurrences of a fixed distance in \cite{OpenAI2026PlanarUnitDistanceBlog}, we would not be surprised to learn that there exists a set of points $P \subset \mathbb R^2$ which contains $n^{2+c}$ isosceles triangles for some constant $c>0$. We now turn to the proof of Theorem \ref{thm:rhomingrid}. 

We prove a slightly more general result, establishing the existence of a large rhombus-free set inside $P$, provided that $P$ does not contain any rich lines. Note that Theorem \ref{thm:rhomingrid} follows immediately from Theorem \ref{thm:rhombigen}, since, for the point set $P=[n] \times [n]$, there is no line $\ell$ such that $|\ell \cap P| > |P|^{1/2}$.

\begin{theorem} \label{thm:rhombigen}
Let $n$ be sufficiently large and let $P \subset \mathbb R^2$ be a set with $|P|=n$ such that no line contains more than $n^{2/3-c}$ elements of $P$ for some absolute constant $c>0$. Then there exists a rhombus-free set $P' \subset P$ of size $|P'| \gg_c n^{2/3}$.
\end{theorem}

\begin{proof} First, apply Corollary \ref{cor:norichcircles} to obtain $Q \subset P$ with $|Q| \geq |P|/2$ and such that no circle centred at an element of $Q$ contains more than $n^{1/2}$ elements of $Q$.

Next, let $\mathcal H$ be the $4$-uniform hypergraph encoding rhombi in $Q$. That is, $\mathcal H$ is the hypergraph with vertex set $Q$ such that four distinct vertices $p,q,s,t \in Q$ form an edge if and only if the points form a rhombus. We will apply Corollary \ref{corl:ind} to find a large independent set in this 4-uniform hypergraph. It follows from Lemma~\ref{lem:rhombicount} that 
\[
\overline{\Delta}_4(\mathcal{H}) \ll n \log n.
\]
Set $\overline{\Delta}_4:=\Theta(n \log n)$ such that $\overline{\Delta}_4(\mathcal{H})\leq \overline{\Delta}_4$. Then, $\overline M:= \overline \Delta_4^{1/3}=\Theta(n \log n)^{1/3}$. 
Since $\mathcal H$ is $4$-uniform, there are no copies of $K_3$ in $\mathcal H$, and so $\Delta_{K_3}(\mathcal{H})=0$.  
We will apply Corollary \ref{corl:ind} with $f=n^c$. For this application to be valid, we need to verify that
\begin{equation} \label{aim34}
    \Delta_{3,4}(\cH) \leq \overline M/f
\end{equation}
and
\begin{equation} \label{aim24}
    \Delta_{2,4}(\cH) \leq \overline M^2/f.
\end{equation}

For \eqref{aim34}, we have the stronger bound $\Delta_{3,4}(\cH) \leq 3$. Indeed, for any three points $p,q,s \in \mathbb R^2$, there exist at most $3$ choices for $t \in \mathbb R^2$ such that $p,q,s,t$ forms a rhombus, and this worst case is realised when $p,q,s$ form an equilateral triangle.

Next, we will verify \eqref{aim24}. Fix $p,q \in Q$. There are two kinds of rhombi that involve $p$ and $q$ that we deal with separately; rhombi with $p$ and $q$ being adjacent, and rhombi such that $p$ and $q$ are on the diagonal. For the first of these possibilities, we need to bound the number of possible choices for $s,t \in Q$ such that
\[
\|p-q\|=\|q-s\|=\|s-t\|=\|t-p\|.
\]
With $p$ and $q$ fixed, it follows from Corollary \ref{cor:norichcircles} that there are at most $n^{1/2}$ choices for $s$. Once $s$ is determined, there are at most $3$ choices for $t$, and so there are at most $3n^{1/2}$ edges involving $p$ and $q$ which correspond to rhombi with $p$ and $q$ adjacent.

For the second of these cases, we fix $p$ and $q$ as a diagonal of the rhombus. Note that $s$ and $t$ must both lie on the perpendicular bisector $b(p,q)$ and be symmetric with respect to the midpoint of $p$ and $q$. Invoking the assumption of the theorem, there are at most $n^{2/3-c}$ choices for $s$, and the choice of $t$ is then determined. It follows that there are at most $n^{2/3-c}$ edges involving $p$ and $q$ which correspond to rhombi with $p$ and $q$ as diagonals. We have therefore shown that
\[
\Delta_{2,4}(\cH) \leq 3n^{1/2} + n^{2/3-c} \leq 4n^{2/3-c} \leq \overline M^2/n^c.
\]
All of the conditions of Corollary \ref{corl:ind} have now been verified, and it follows that
\[
\alpha(\cH) \gg |Q| \left(\frac{c  \log n}{n \log n} \right )^{1/3} = \Omega_c (n^{2/3}).
\]
The proof is now complete, since any independent set in $\cH$ corresponds to a rhombus-free set.
\end{proof}

\section{Concluding remarks}
Note that most of the proof of Theorem \ref{thm:rhombigen} works for an arbitrary point set, and that the condition that the point set avoids rich lines is only required to bound $\Delta_{2,4}(\cH)$. We suspect that it is possible to use the techniques from this paper to remove this condition from Theorem \ref{thm:rhombigen}. In particular, one can define the hypergraph encoding rhombi more carefully, introducing $2$-edges corresponding to points with rich bisectors. However, we could not make this argument work, and the best we can do is to prove that any $P \subset \mathbb R^2$ with $|P|=n$ contains a rhombus-free subset $P' \subset P$ such that
\[
|P'| \gg 
n^{2/3}\left(\frac{\log\log n}{\log n}\right)^{1/3},
\]
thus slightly improving on the lower bound which follows from Lemma \ref{lem:rhombicount}. We omit the details of the proof of this result.

\section{Acknowledgments}
Research of Felix Christian Clemen was supported by a PIMS Postdoctoral Fellowship PIMS-20260602-PDF. Jakob Führer and Oliver Roche-Newton were supported by the Austrian Science Fund (FWF) under the project PAT2559123. We thank Thomas Bloom, Dylan King, Jonathan Noel, Marcelo Sales and Audie Warren for discussions on related subset selection problems.

\bibliographystyle{abbrv} 
\bibliography{bibilo}

\end{document}